\documentclass[11pt]{amsart}
\usepackage[sorted-cites]{amsrefs}
\usepackage{verbatim}
\usepackage{color}
\usepackage{amscd}
\renewcommand{\bar}{\overline}

\newcommand{\rr}{\mathbb{R}}
\newfont{\fnt}{cmr10 scaled 550}

\newtheorem{theorem}{Theorem}

\newtheorem{lemma}{Lemma}
\newtheorem{cor}{Corollary}
\newtheorem{prop}{Proposition}

\newtheorem{definition}{Definition}

\newtheorem{example}{Example}
\theoremstyle{remark}
\newtheorem{remark}{Remark}

\usepackage{amsmath,amsthm,amscd}

\newcommand{\Ric}{\textrm{Ric}}

\begin{document}
%\tableofcontents

\title[Stability properties  and gap theorem]{Stability properties  and gap theorem for complete $f$-minimal hypersurfaces}

%\date{November 22, 2009}

 \subjclass[2000]{Primary: 58J50;
Secondary: 58E30}

\thanks{The authors were  partially supported by CNPq and Faperj of Brazil.}

\author{Xu Cheng}
\address{Instituto de Matematica e Estat\'\i stica, Universidade Federal Fluminense,
Niter\'oi, RJ 24020, Brazil, email: xcheng@impa.br}

\author[Detang Zhou]{Detang Zhou}
\address{Instituto de Matematica e Estat\'\i stica, Universidade Federal Fluminense,
Niter\'oi, RJ 24020, Brazil, email: zhou@impa.br}

%\email{zhou@impa.br}

\newcommand{\M}{\mathcal M}

\begin{abstract}   In this paper, we study complete  oriented $f$-minimal hypersurfaces  properly immersed  in a cylinder shrinking soliton $(\mathbb{S}^n\times \mathbb{R}, \overline{g}, f)$. We prove that such hypersurface with $L_f$-index one must be either $\mathbb{S}^n\times\{0\}$ or $\mathbb{S}^{n-1}\times\mathbb{R}$, where $\mathbb{S}^{n-1}$  denotes the  sphere in $\mathbb{S}^{n}$ of the same radius. Also we  prove a pinching theorem  for them.

\end{abstract}

\maketitle
\baselineskip=18pt

%\section{Introduction}\label{introduction}

%\section{Definitions and notation}\label{notation}

%%%%%%%%%SECTION%%%%%%%%%%%

%%%%%%%%%%%%%%%%SECTION%%%%%%%%%%%%%%%
\section{Introduction}
The $f$-minimal hypersurfaces in smooth metric measure  spaces are  the generalization of self-shrinkers in the Euclidean space and   have been studied recently.   See, for instance, \cite{B}, \cite{CMZ1}, \cite{CMZ2}, \cite{CMZ3}, \cite{E}, \cite{Fa}, \cite{IR},  \cite{LW}.  Recall a  hypersurface  $(\Sigma,g)$ isometrically  immersed in a Riemannian manifold $(M,\bar{g})$ is called an $f$-minimal hypersurface if its   mean curvature $H$ satisfies that,  $$H=\langle \overline{\nabla} f,\nu\rangle,$$  where $\nu$ is the unit normal to $\Sigma$, $f$ is a smooth function on $M$, and $ \overline{\nabla} f$ denotes the gradient of $f$ on $M$. It is known that  an $f$-minimal  hypersurface can be viewed in  two basic ways:  1)  it is  a critical point of the weighted volume functional $\int_\Sigma e^{-f}d\sigma$ of $\Sigma$,  where $d\sigma$ denote the volume element of $(\Sigma, g)$; 2)  it is a minimal hypersurface in  $(M, \tilde{g})$, where the new metric $\tilde{g}=e^{-\frac{2}{n}f}\bar {g}$ of $M$ is conformal to $\bar{g}$.  Besides, $f$-minimal hypersurfaces appear  in the study of mean curvature flow of a hypersurfce in an  ambient manifold evolving by  Ricci flow. Recently
Lott \cite{Lo} and Magni-Mantegazza-Tsats \cite{MMT}    proved that Huisken's monotonicity formula holds  when the ambient is a   gradient Ricci soliton solution to the Ricci flow.  Lott  \cite{Lo} introduced the  concept of mean curvature soliton  for the mean curvature flow of a hypersurface in a    gradient Ricci soliton solution. By its definition,  a mean curvature soliton is just an  $f$-minimal hypersurface,  where  $f$ is the potential function of the ambient gradient Ricci soliton.   Some compactness theorem of $f$-minimal surfaces in three-dimensional  manifolds were proved by Cheng-Mejia-Zhou   \cite{CMZ1}, \cite{CMZ2}.  

Motivated by the above facts, Cheng-Mejia-Zhou \cite{CMZ3} studied  $f$-minimal hypersurfaces in a cylinder shrinking soliton of type $(\mathbb{S}^n\times \mathbb{R}, \overline{g}, f)$ with    the potential function $f$.
In \cite{CMZ3}, the authors  studied the drifted Laplacian $\Delta_f=\Delta f-\langle {\nabla}f, \nabla\cdot\rangle$ of  geometric quantities for $f$-minimal hypersurfaces  and derived a Simons' type equation and some other equations. These equations  involve  the Barky-\'Emery Ricci curvature $\overline{Ric}_f:=\overline{Ric}+\overline{\nabla}^2f$ of the ambient manifold and   have simpler expressions when the ambient manifold is  a gradient Ricci soliton. 
Furthermore,   the stability properties and gap phenomenon of
closed, i.e.,  compact and  without boundary,  $f$-minimal hypersurfaces   immersed in a cylinder shrinking soliton of type $(\mathbb{S}^n\times \mathbb{R}, \overline{g}, f)$  were studied. Especially, 
the classification  of the closed immersed $f$-minimal hypersurfaces with $L_f$-index one and  a pinching theorem of the norm  $|A|^2$ of the second fundamental form were obtained (see the definitions of $L_f$ operator and $L_f$-index in Section \ref{sec-notation}).  In this paper, we will discuss the complete noncompact case. 
First, we prove the following

\begin{theorem} \label{th-index}   Let  $\Sigma$ be a  complete  oriented $f$-minimal hypersurface  properly immersed in
the  cylinder shrinking soliton $(\mathbb{S}^{n}(\sqrt{2(n-1)})\times\mathbb{R}, \overline{g}, f)$. Then $L_f$-index of $\Sigma$ satisfies that $L_f$-$\textrm{ind}(\Sigma)\geq 1$.

Moreover $L_f$-index of $\Sigma$ is one   if and only if $\Sigma$  is  either $\mathbb{S}^n(\sqrt{2(n-1)})\times\{0\}$ or $\mathbb{S}^{n-1}(\sqrt{2(n-1)})\times\mathbb{R}$, where $\mathbb{S}^{n-1}(\sqrt{2(n-1)})$  denotes the totally geodesic sphere of  $\mathbb{S}^{n}(\sqrt{2(n-1)})$.
\end{theorem}

\noindent  Here $\mathbb{S}^n(\sqrt{2(n-1)})$ denote the  round sphere in the Euclidean space $\mathbb{R}^{n+1}$ of radius $\sqrt{2(n-1)}$ centered at the origin. The cylinder shrinking soliton  $(\mathbb{S}^{n}(\sqrt{2(n-1)})\times\mathbb{R},\overline{g}, f)$ is a triple satisfying that
 it has the product metric $\overline{g}$ of $(\mathbb{S}^{n}(\sqrt{2(n-1)})$ and $\mathbb{R}$,  $f(x,h)=\frac {h^2}{4}, (x,h)\in (\mathbb{S}^{n}(\sqrt{2(n-1)})\times\mathbb{R}$ (see more details in Section \ref{sec-notation}).

\noindent  Secondly, we prove a pinching theorem  as follows:
 
 \begin{theorem}\label{th-pinch}
Let   $\Sigma$ be a complete oriented $f$-minimal hypersurface properly immersed in the  cylinder shrinking soliton $(\mathbb{S}^{n}(\sqrt{2(n-1)})\times\mathbb{R}, \overline{g}, f)$,  $n\geq 3$. Assume that the norm $|A|$ of the second fundamental form of $\Sigma$ satisfies
\begin{equation}\label{pinch-1}
\left| |A|^2-\frac14\right|\leq \frac14\biggr(\sqrt{1-\frac{8}{n-1}\alpha^2(1-\alpha^2)}\biggr),
\end{equation}
where $\alpha=\langle\frac{\partial}{\partial h}, \nu\rangle$,  $h$ denotes the coordinate of the second factor $\mathbb{R}$ in $(\mathbb{S}^{n}(\sqrt{2(n-1)})\times\mathbb{R}$.
Then $\Sigma$ must be one of the following three cases: 
\begin{itemize}
\item  $\mathbb{S}^{n}(\sqrt{2(n-1)})\times\{0\}$, 
\item $\mathbb{S}^{n-1}(\sqrt{2(n-1)})\times\mathbb{R}$, and
\item  $T\times \mathbb{R}$, where $T$ denotes the minimal Clifford torus $\mathbb{S}^k(\sqrt{2k})\times\mathbb{S}^l(\sqrt{2l})\subset \mathbb{S}^n(\sqrt{2(n-1)})$ in $ \mathbb{S}^{n}(\sqrt{2(n-1)})$,  $k+l=n-1$.
\end{itemize}
 \end{theorem}
\noindent Theorem \ref{th-pinch}   implies   that 
\begin{cor}\label{cor-gap}There is no any complete oriented  $f$-minimal  hypersurface properly immersed  in the  cylinder shrinking soliton $(\mathbb{S}^{n}\bigr(\sqrt{2(n-1)}\bigr)\times\mathbb{R}, \overline{g}, f)$, $n\geq3$, so that its  norm $|A|$ of second fundamental form  satisfies 
\begin{equation*}
\frac14\biggr(1-\sqrt{1-\frac{2}{n-1}}\biggr)\leq |A|^2\leq \frac14\biggr(1+\sqrt{1-\frac{2}{n-1}}\biggr).
\end{equation*}
\end{cor}

 \begin{remark} The cases of closed $f$-minimal hypersurfaces in Theorems \ref{th-index} and \ref{th-pinch} and Corollary \ref{cor-gap}  were proved in \cite{CMZ3} (see \cite{CMZ3} Th.1, Th.2, and Corollary 1). 
 \end{remark}
 For an immersed  $f$-minimal hypersurface  in a shrinking gradient Ricci soliton with certain condition on $f$, the properness of  its immersion,  its  polynomial volume growth, and  its  finite weighted volume are equivalent each other (\cite{CMZ2}).  This guarantees  integrability of some weighted integrals. Furthermore if there are  some integrability conditions on  $|A|^2$, one may obtain the integral identities, which are derived from Simons' type equation and other identities  (Propositions \ref{lem-alpha} and \ref{lem-A}). As an immediate application of the integral identities,  we  prove the rigidity of complete totally geodesic $f$-minimal hypersurfaces in $(\mathbb{S}^{n}\bigr(\sqrt{2(n-1)}\bigr)\times\mathbb{R}, \overline{g}, f)$ as follows:
 \begin{theorem}\label{lem-geodesic} Let $\Sigma$ be a complete oriented $f$-minimal hypersurface properly  immersed in $(\mathbb{S}^{n}(\sqrt{2(n-1)})\times\mathbb{R},\bar{g}, f)$. Then $\Sigma$ is totally geodesic if and only if it  is  either $\mathbb{S}^n(\sqrt{2(n-1)})\times\{0\}$ or $\mathbb{S}^{n-1}(\sqrt{2(n-1)})\times\mathbb{R}$.
\end{theorem}
 
 When $f$-minimal hypersurfaces has $L_f$-index one, the results on the spectrum of the operator $L_f$ imply the integrability of $|A|^2$ (Lemma \ref{A-finite}). Hence the adapted integral identity can be used in order to make classification (Theorem \ref{th-index}). Also the pinching condition of $|A|^2$ make the use of integral identities possible in the proof of Theorem \ref{th-pinch}.
 
We  mentioned that for self-shrinkers, the stability properties  were studied by Colding-Minicozzi \cite{CM3} and Hussey \cite{H}, and gap phenomenon was discussed, for instance, by Le-Sesum \cite{LS} and  Cao-Li \cite{CL}. 

The rest of this paper is organized as follows: In Section \ref{sec-notation}, we give some notations and conventions as preliminaries; In Section \ref{sec-weight}, we give some spectrum properties of the operator $\Delta_f+q$ which will be used in Section \ref{sec-index};  In Section \ref{sec-integral}, we prove some integral identities for $f$-minimal hypersurfaces and  Theorem 3; In Section \ref{sec-index}, we study  $L_f$-index and prove Theorem \ref{th-index};  In Section \ref{sec-pinch}, the proof of Theorem \ref{th-pinch} is given.

%%%%%%%%%%%%%%%%%%%%%%%%%SECTION%%%%%%%%%%%%%%%%%%%%%%%%%

\section{Notations and conventions}\label{sec-notation}

 Throughout this paper, we will use the same conventions and notations as in \cite{CMZ3}, unless otherwise specified. For instance,  $(M^{n+1}, \overline{g}, e^{-f}d\mu)$  denotes a smooth metric measure space, where $f$ is a smooth function on $M$ and $d\mu$ is the volume element of $M$ induced by the metric $\bar{g}$.  $(\Sigma, g)$ denotes a hypersurface isometrically immersed in $(M,\overline{g})$ respectively.  We denote by a bar all quantities on $(M, \bar{g})$, for instance by $\overline{\nabla}$,  $\overline{\textrm{Ric}}$, and $\overline{\text{Ric}}_{f}$
 the Levi-Civita connection,  the  Ricci curvature tensor, and the Bakry-\'Emery Ricci curvature of $(M, \bar{g})$ respectively.  We denote still by  $f$  the restriction of $f$ on $\Sigma$.
Also we denote for instance by ${\nabla}$,  $\textrm{Ric}$, $\Delta$, and $d\sigma$, 
 the Levi-Civita connection, the  Ricci curvature tensor,  the Laplacian,  and the volume element of $(\Sigma, g)$ respectively.
 The second fundamental form $A$  of $\Sigma$ is defined  by 
$$A: T_p\Sigma\times T_p\Sigma\to \mathbb{R}, \quad  A(X,Y)=\langle\overline{\nabla}_X \nu, Y\rangle, $$
\noindent where $p\in \Sigma, X, Y\in T_p \Sigma$, $\nu$
is a unit normal vector at $p$. In a local orthonormal  system $\{e_i\}, i=1,\ldots, n$ of $\Sigma$, the components of $A$ are denoted by $a_{ij}=A(e_i,e_j)=\langle\overline{\nabla}_{e_i}\nu, e_j\rangle$.
The mean curvature $H$ and the weighted mean curvature of $\Sigma$ is  defined, respectively,  by \begin{equation*}\begin{split}
H&=\text{tr}A=\displaystyle\sum_{i=1}^na_{ii},\\
H_f&=H-\langle \overline{\nabla} f,\nu\rangle.
\end{split}
\end{equation*}

\noindent   $(\Sigma,g)$ is called  $f$-minimal  if its   mean curvature $H$ satisfies that
$$H=\langle \overline{\nabla} f,\nu\rangle,$$ 
 for any $p\in \Sigma$,  
 where $\nu$ is the unit normal to $\Sigma$. 
 
 The drifted Laplacian $\Delta_f$ and  $L_f$ operator on $\Sigma$ are  defined by 
\begin{align*}\Delta_f&=\Delta-\langle\nabla f,\nabla\cdot\rangle,\\
L_f&=\Delta_f+|A|^2+\overline{Ric}_f(\nu,\nu).
\end{align*}  
It holds that  for $u\in C_0^{1}(\Sigma)$,  $w\in C^2(\Sigma)$,
$$\int_{\Sigma} u\Delta_fwe^{-f}dv=-\int_{\Sigma}\langle\nabla u,\nabla w\rangle e^{-f}dv.$$
Since $\Delta_f$ is self-adjoint in the weighted space $L^2(e^{-f}d\sigma)$, we may define a symmetric bilinear form on the space $C_0^{\infty}(\Sigma)$  of compactly supported smooth functions on $\Sigma$ by, for $\phi, \psi\in C_0^{\infty}(\Sigma)$,
\begin{equation}\begin{split}
B_f(\phi, \psi):&=-\int_\Sigma \phi L_f\psi e^{-f}d\sigma\\
&=\int_{\Sigma}[\langle\nabla\phi,\nabla \psi\rangle-(|A|^{2}+\overline{\textrm{Ric}}_{f}(\nu,\nu))\phi\psi]e^{-f}d\sigma.
\end{split}
\end{equation}

Now suppose that $(\Sigma, g)$ is an $f$-minimal hypersurface.   $\Sigma$ is called $L_f$-stable if $B_f(\phi,\phi)\geq 0$ for all $\phi\in C_0^{\infty}(\Sigma)$.
 It is known that $L_f$-stability $\Sigma$ means that the weighted volume $\int_{\Sigma}e^{-f}d\sigma$  of  $\Sigma$ is locally minimizing. Furthermore, we may define the $L_f$-index of $\Sigma$ analogous to defining the (Morse) index of  the operator $\Delta+q$  (cf \cite{F}).
  Consider the Dirichlet eigenvalue problems of $L_f$ on a compact domain $\Omega\subset \Sigma$: 
 $$L_fu+\lambda u=0, \quad u\in\Omega;  \quad u|_{\partial \Omega}=0.$$
 
\noindent  It is known that the $L^2(e^{-f}d\sigma)$-spectrum of the operator $L_f$ on $\Omega$ is discrete and the set of all (Dirichlet) eigenvalues, counted with multiplicity,  is an increasing sequence $$\lambda_1(\Omega)< \lambda_2(\Omega)\leq\cdots$$ with $\lambda_i(\Omega)\to \infty$ as $i\to \infty$. 

The $L_f$-index of $\Omega$, denoted by $L_f$-ind$(\Omega)$, is defined as the number of negative  Dirichlet eigenvalues of $L_f$ on $\Omega$ counted with multiplicity.
 The variational characterization of eigenvalues implies that if $\Omega_1\subset \Omega_2$, then $L_f$-ind$(\Omega_1)\leq$ $ L_f$-ind$(\Omega_2)$. Hence one may define 
 the  $L_{f}$-index of $\Sigma$ as follows.
 \begin{definition} \label{def-index} The $L_{f}$-index of $\Sigma$, denoted by  $L_f$-ind$(\Sigma)$,  is defined to be the  supremum  over  compact domains of $\Sigma$ of   $L_f$-index of compact domain, that is,
  $L_f$-ind$(\Sigma)=\displaystyle\sup_{\Omega\subset\subset\Sigma}$$L_f$-ind$(\Omega)$.
 \end{definition} 
 \noindent  By Definition \ref{def-index},  $L_f$-ind$(\Sigma)=0$ if and only if $\Sigma$ is $ L_f$-stable.
  
  If $\Sigma$ is an $n$-dimensional hypersurface in a complete manifold $M^{n+1}$,  $\Sigma$ is said to have polynomial volume growth  if, for a $p\in M$ fixed, there exist constants $C$ and $d$ so that for all $r\geq 1$,
\begin{equation} \text{Vol} ({B}^M_r( p)\cap \Sigma)\leq Cr^d \label{P-1-1},
\end{equation}
where ${B}^M_r( p)$ is the extrinsic ball of radius $r$ centered at $p$,  $ \text{Vol}  ({B}^M_r( p)\cap \Sigma)$ denotes the volume of ${B}^M_r( p)\cap \Sigma$.
When $d=n$ in (\ref{P-1-1}), $\Sigma$ is said to be of  Euclidean volume growth.

Now we recall the definition and some related facts of gradient Ricci solitons (cf Chapter 4 of \cite{CLN}, or \cite{C}). A quadruple  $(M, \overline{g}, f, \rho)$ is called a gradient Ricci soliton if it satisfies that 
$$\overline{Ric}+\overline{\nabla}^2f=\frac{\rho}{2}\bar{g},$$
where $\rho$ is a constant. If $\rho = 0$,  $\rho> 0$,  or $\rho < 0$, the gradient Ricci soliton $(M, \overline{g}, f, \rho)$ is called  steady,  shrinking or expanding   respectively.  

It is known that  a gradient Ricci soliton $(M, \overline{g}, f, \rho)$ has an associated time-dependent version, i.e.,  a gradient Ricci soliton solution to the Ricci flow. By abuse of notations,  a gradient Ricci soliton solution to the Ricci flow is often called gradient Ricci soliton if there is  no confusion. 

 Cylinder shrinking solitons of type $\mathbb{S}^n\times \mathbb{R}$, $n\geq 2$:  Let  $\mathbb{S}^n(\sqrt{2(n-1)})$ denote the  round sphere in the Euclidean space $\mathbb{R}^{n+1}$ of radius $\sqrt{2(n-1)}$ centered at the origin. Consider  the triple $(\mathbb{S}^{n}(\sqrt{2(n-1)})\times \mathbb{R}, \bar{g}(t), f(t))$, $t\in (-\infty,0)$,  with  the metric 
$$\overline{g}(t)=2(n-1)|t|ds_{n}^2+dh^2, $$
and $$f(t)(x,h)=\frac{h^2}{4|t|}, \quad (x,h)\in \mathbb{S}^{n}(\sqrt{2(n-1)})\times\mathbb{R}.$$
where $ds_n^2$ and $dh^2$ denote the canonical metrics on the unit round sphere in $\mathbb{R}^{n+1}$ and $\mathbb{R}$ respectively. One may verify that
 $$\overline{Ric}_{f(t)}=\overline{Ric}+\overline{\nabla}^2(f(t))=\frac1{2|t|}\bar{g}(t).$$
 Hence for each $t$,  $(\mathbb{S}^n(\sqrt{2(n-1)})\times\mathbb{R}, \overline{g}(t), f(t))$ is a gradient shrinking Ricci soliton and they are the examples of cylinder shrinking solitons of type $\mathbb{S}^n\times \mathbb{R}$.  Moreover, it is known that $(\mathbb{S}^n(\sqrt{2(n-1)})\times\mathbb{R}, \overline{g}(t), f(t))$, $t\in (-\infty,0)$, is the gradient shrinking Ricci soliton solution to the Ricci flow $\frac{\partial }{\partial t}\bar{g}=-2\overline{Ric}$.
 
In this paper, we denote by  $(\mathbb{S}^{n}(\sqrt{2(n-1)})\times \mathbb{R}, \bar{g}, f)$  the  cylinder shrinking soliton $(\mathbb{S}^n(\sqrt{2(n-1)})\times\mathbb{R}, \overline{g}(-1), f(-1))$. To study the $f$-minimal hypersurfaces in a cylinder shrinking soliton of type $\mathbb{S}^n\times \mathbb{R}$,  it suffices to  consider the ambient  is $(\mathbb{S}^{n}(\sqrt{2(n-1)})\times \mathbb{R}, \bar{g}, f)$.

We also need   the following result:
\begin{prop}\label{equiv}(\cite{CMZ2}, Corollary 1) Let $(M^{m},\bar{g},f)$ be a complete shrinking gradient Ricci soliton with $\overline{\Ric}_{f}= \frac12\bar{g}$. Assume that   $f$ is  a convex 
function, i.e., $\overline{\nabla}^2f\geq 0$,  and satisfies  $|\overline{\nabla}f|^{2}\leq  f$.  If $\Sigma$ is an immersed complete $f$-minimal submanifold in $M$, then  for  $\Sigma$ the properness of immersion, polynomial volume growth, and finite weighted volume are equivalent.
\end{prop}

Throughout this paper, we assume that $f$-minimal hypersurfaces are orientable.

%%%%%%%%%%%%%%%%%%%%%%%%%SECTION%%%%%%%%%%%%%%%%%%%%%%%%%

\section{$L_f$-index of  a weighted manifold}\label{sec-weight}

In this section,  suppose that $(\Sigma, g, e^{-f}d\sigma)$ is  a complete smooth metric measure space. Here $\Sigma$  is not necessarily a hypersurface. Let $q(x)$ be a continuous function on $\Sigma$ and consider  the  operator
$L_f=\Delta_f+q(x)$ on $\Sigma$.  Recall that the bottom of the spectrum of the operator $L_f$, denoted by $\mu_1(L_f)$, satisfies that
$$\mu_{1}(L_f)=\inf_{ u\not\equiv 0, u\in C_0^{\infty}(\Sigma)}\frac{\displaystyle\int_{\Sigma}(|\nabla u|^{2}-qu^2)e^{-f}d\sigma}{\displaystyle\int_{\Sigma}u^{2}e^{-f}d\sigma}.$$
Just substitute  $q(x)$ for $|A|^2+\overline{Ric}_f(\nu,\nu)$ in   Section \ref{sec-notation}, one can similarly define
the $L_f$-index of $\Sigma$, denoted by $L_{f}$-ind$(\Sigma)$,  to be the  supremum over compact domains of $\Sigma$ of the number of negative  (Dirichlet) eigenvalues of $L_f$, In \cite{F}, Fischer-Colbrie showed that if a complete minimal hypersurface has finite induex, then it is stable outside of a compact set and there is a positive function $u$ on it so that  $Ju=(\Delta+|A|^2+\overline{Ric}(\nu,\nu))u=0$ outside of the compact set. Moreover, she proved 
an equivalent statement of finite index of a complete minimal hypersurface (\cite{F}, Proposition 2) and also proved that if the index of a complete minimal hypersurface is finite, then the bottom of the spectrum of stability operator $J$ is finite (see (4) in \cite{F}). An analogous argument gives the following weighted version for $L_f$: 
\begin{prop} \label{mu1}  The following  are equivalent:

(i) $\Sigma$ has finite $L_f$-index;

(ii) There exists a finite dimensional subspace $W$ of the weighted space $L^2(e^{-f}d\sigma)$ having an orthonormal basis $\psi_1,\ldots,\psi_k$ consisting of eigenfunctions with eigenvalues $\lambda_1,\ldots,\lambda_k$ respectively. Each $\lambda_i$ is negative and  for $\phi\in C_0^{\infty}(\Sigma)\cap W^{\perp}$, $Q(\phi,\phi)\geq 0$.

 Moreover if the $L_f$-ind$(\Sigma)<\infty$, then $L_f$-ind$(\Sigma)=dim W$ and the bottom $\mu_1$ of the spectrum of $L_f$ is  finite. Furthermore if $1\leq L_f$-ind$(\Sigma)<\infty$, $\mu_1$ is the least  negative $L^2(e^{-f}d\sigma)$ eigenvalue.
\end{prop}

 Denote by  $W^{1,2}(e^{-f}d\sigma)$ the weighted Sobolev space which is the set of the functions $u$ on $ \Sigma$ satisfying that $\int_{\Sigma}(u^2+|\nabla u|^2)e^{-f}d\sigma<\infty$.  $W^{1,2}(e^{-f}d\sigma)$ has the norm:
$$\|u\|_{W^{1,2}(e^{-f}d\sigma)}:=\left(\int_{\Sigma}(u^2+|\nabla u|^2)e^{-f}d\sigma\right)^{\frac12}.$$
Through a similar proof, we can extend Lemmas 9.15 and 9.25 in \cite{CM3} for the operator $L=\Delta-\frac 12\langle x, \nabla\cdot\rangle+|A|^2+\frac 12$ on hypersurfaces in $\mathbb{R}^{n+1}$ to  the operator $L_f$. 
\begin{prop}\label{eigenfunction1} If $\mu_1(L_f)\neq -\infty$, then there exists a positive $C^2$ function $u$ on $\Sigma$ with $L_fu=-\mu_1(L_f) u$. Moreover, if $w$ is in the the weighted $W^{1,2}(e^{-f}d\sigma)$ space and $L_fw=-\mu_1(L_f) w$, then $w=Cu$ for some $C\in\mathbb{R}$.
\end{prop}

\begin{prop} \label{logf} Suppose that $h$ is a $C^2$ function with $L_fh=-\mu h$, $\mu\in\mathbb{R}$. If $h>0$ and $\phi$ is in  $W^{1,2}(e^{-f}d\sigma)$, then 
\begin{equation}
\label{A-est} \int_{\Sigma}\phi^2(2q+|\log h|^2)e^{-f}d\sigma\leq \int_{\Sigma}(4|\nabla\phi|^2-2\mu\phi^2)e^{-f}d\sigma.
\end{equation}
\end{prop}

%%%%%%%%%%%%%%%%%%%%%%%%%SECTION%%%%%%%%%%%%%%%%%%%%%%%%%

\section{Integral identities of  complete $f$-minimal hypersurfaces}\label{sec-integral}

From this section, we start to consider  complete  $f$-minimal hypersurfaces immersed in  the cylinder shrinking soliton  $(\mathbb{S}^{n}(\sqrt{2(n-1)})\times\mathbb{R},\bar{g}, f)$, $n\geq 2$ with $f(x,h)=\frac{h^2}{4}$,  $(x,h)\in \mathbb{S}^{n}(\sqrt{2(n-1)})\times\mathbb{R}$ and    the  metric 
\begin{equation*}
\overline{g}=g_{\mathbb{S}^{n}(\sqrt{2(n-1)})}+dh^{2}.
\end{equation*}
 
 $(\mathbb{S}^{n}(\sqrt{2(n-1)})\times\mathbb{R},\overline{g},f)$ associates to a smooth metric measure space $(\mathbb{S}^{n}(\sqrt{2(n-1)})\times\mathbb{R},\overline{g},e^{-f}d\mu)$.  It is known (cf \cite{CMZ3}) that
\begin{align}
&\overline{\nabla}f=\frac h2\frac{\partial }{\partial h},  \quad \overline{\nabla}^2f=\frac12\frac{\partial}{\partial h}\otimes\frac{\partial}{\partial h}\nonumber\\
&\overline{\textrm{Ric}}=\frac{1}{2}\overline{g}-\frac12\frac{\partial}{\partial h}\otimes\frac{\partial}{\partial h},\nonumber\\
\label{ric-f}
&\overline{\textrm{Ric}}_{f}=\frac12\overline{ g}.
\end{align}

  \noindent 
For an $f$-minimal hypersurface $\Sigma$ immersed  in $\mathbb{S}^{n}(\sqrt{2(n-1)})\times \mathbb{R}$,  
$$0=H_f=H-\frac h2\langle \frac{\partial}{\partial h},\nu\rangle=H-\frac{h}{2}\alpha,$$
 where $\alpha=\langle \frac{\partial}{\partial h},\nu\rangle$. Hence $\Sigma$ satisfies 
 \begin{equation}\label{eq-H}
 H=\frac{h\alpha}{2}.
 \end{equation}
 Under a local orthonormal frame $\{e_i\}_{i=1}^n$ on $\Sigma$,
 \begin{align*}
\nabla_{e_i}\alpha&=\langle\overline{\nabla}_{e_i}\nu,\frac{\partial}{\partial h}\rangle=\sum_{j=1}^na_{ij}\langle e_j,\frac{\partial}{\partial h}\rangle.
\end{align*}
\begin{equation}\label{ine-nabla}
|\nabla\alpha|^2=\displaystyle\sum_{i=1}^n|\nabla_{e_i}\alpha|^2\leq \displaystyle\sum_{i=1}^n\left(\displaystyle\sum_{j=1}^na_{ij}^2\right)\left(\displaystyle\sum_{j=1}^n\langle e_j,\frac{\partial}{\partial h}\rangle^2\right)\leq |A|^2.
\end{equation}
 The operator $L_f$ on $\Sigma$ is equal to
\begin{equation}L_f=\Delta-\frac{h}{2}\langle \left(\frac{\partial}{\partial h}\right)^{T},\nabla\cdot\rangle+|A|^2+\frac{1}{2}.
\end{equation}
Now we give some examples of $f$-minimal hypersurfaces in $(\mathbb{S}^{n}(\sqrt{2(n-1)})\times\mathbb{R},\bar{g}, f)$.

\begin{example}\label{ex-fmin-1} (\cite{CMZ3}) Lemma 1) The slice $\mathbb{S}^{n}(\sqrt{2(n-1)})\times \{0\}$ is $f$-minimal and totally geodesic. Furthermore a complete immersed $f$-minimal hypersurface  is   in a horizontal slice $\mathbb{S}^{n}(\sqrt{2(n-1)})\times\{h\}$,  $h\in \mathbb{R}$ fixed,  if and only if it is $ \mathbb{S}^{n}(\sqrt{2(n-1)})\times \{0\}$.
\end{example}

\begin{example}  \label{ex-fmin-2} Assume that $\Sigma_1$ is an immersed  hypersurface in $\mathbb{S}^{n}(\sqrt{2(n-1)})$. Then the product $\Sigma=\Sigma_1\times \mathbb{R}$ is $f$-minimal in $(\mathbb{S}^{n}(\sqrt{2(n-1)})\times\mathbb{R},\bar{g}, f)$  if and only if  $\Sigma_1$ is   minimal  in $\mathbb{S}^{n}(\sqrt{2(n-1)})$.
 Particularly, the totally geodesic hypersurface $\mathbb{S}^{n-1}(\sqrt{2(n-1)})\times \mathbb{R}$ is  $f$-minimal  in $\mathbb{S}^{n}(\sqrt{2(n-1)})\times \mathbb{R}$, where   the $(n-1)$-dimensional sphere $\mathbb{S}^{n-1}(\sqrt{2(n-1)})\subset\mathbb{S}^{n}(\sqrt{2(n-1)})$.

The reason is that:  the unit normal $\nu$ of $\Sigma$ is also the unit normal of $\Sigma_1$ and   $\nu\in T\mathbb{S}^n(\sqrt{2(n-1)})$. Then
$\langle\overline{\nabla}f,\nu\rangle=0$ and $\overline{\nabla}_{\frac{\partial}{\partial h}}\nu=0$.
Hence $H_f(x,h)=H(x,h)=H_{\Sigma_1}(x)$, $(x,h)\in \Sigma_1\times\mathbb{R}$.  This implies that $H_f=0$ on $\Sigma$ if and only if $\Sigma_1$ is minimal in $\mathbb{S}^{n}(\sqrt{2(n-1)})$.
\end{example}

\noindent Note that  $\overline{\nabla}^2f\geq 0$   and $|\overline{\nabla}f|^2=f=\frac{h^2}{4}$. Applying Proposition \ref{equiv},
    we get that 
    \begin{prop}\label{prop-1-cor} For  any  complete immersed $f$-minimal hypersurface $\Sigma$ in $(\mathbb{S}^{n}(\sqrt{2(n-1)})\times\mathbb{R},\bar{g}, f)$, 
 the properness of immersion, polynomial volume growth, and finite weighted volume are 
 equivalent.
 \end{prop}
 \begin{remark}  Proposition \ref{prop-1-cor} will be used frequently in the proofs of this paper without mentioned. For self-shrinkers in $\mathbb{R}^{n+1}$,  the statement that properly immersed self-shrinkers in $\mathbb{R}^{n+1}$ implies the Euclidean volume growth and finite weighted\ volume was proved in \cite{DX}.  The equivalence of properness of immersion, finite weighted volume, polynomial volume growth was proved in \cite{CZ1}.

 \end{remark}

A Simons' type equation and the other identities for $f$-minimal hypersurfaces immersed in $(\mathbb{S}^{n}(\sqrt{2(n-1)})\times\mathbb{R},\bar{g}, f)$ were derived in  \cite{CMZ3} as follows:
\begin{prop}\label{lem-1}(\cite{CMZ3}) Let $\Sigma$ be an $f$-minimal hypersurface immersed in $(\mathbb{S}^{n}(\sqrt{2(n-1)})\times\mathbb{R},\bar{g}, f)$. Then
\begin{align}
\label{delta-0} L_f\alpha&=\frac12\alpha,\\
\label{delta-1}\frac{1}{2}\Delta_f\alpha^{2}&=|\nabla\alpha|^{2}-|A|^{2}\alpha^{2},\\
\label{delta-2}\frac{1}{2}\Delta_fH^{2}&=|\nabla H|^{2}-(|A|^{2}+\frac12)H^{2}+\frac12\langle\nabla\alpha^{2},\nabla f\rangle,\\
\label{delta-3}\frac{1}{2}\Delta_f|A|^{2}&=|\nabla A|^{2}+|A|^{2}(\frac12-|A|^{2})-\frac{1}{n-1}(|\nabla\alpha|^{2}-\alpha^{2}|A|^{2})\\
&\quad -\frac{1}{n-1}(\alpha^{2}f-\langle\nabla\alpha^{2},\nabla f\rangle).\nonumber
\end{align}
\end{prop}

\noindent  Now we prove the identities in Proposition \ref{lem-1} imply some  integral identities for complete $f$-minimal hypersurfaces  (the case of closed $f$-minimal hypersurfaces was proved in \cite{CMZ3}). We denote by  $p\in \Sigma$ and $B_j$  a fixed point  and  the geodesic sphere of $\Sigma$ of radius $j$ centered at $p$, respectively.  Let $\varphi_j$  be the nonnegative cut-off functions satisfying  that  $\varphi_j$ is $1$ on $B_j$,  $|\nabla\varphi_j|\leq 2$ on $B_{j+1}\setminus B_j$, and $\varphi_j=0$ on $\Sigma\setminus B_{j+1}$.
\begin{prop}\label{lem-alpha}
Let $\Sigma$ be a complete oriented $f$-minimal hypersurface  properly immersed in $(\mathbb{S}^{n}(\sqrt{2(n-1)})\times\mathbb{R},\bar{g}, f)$.  Assume that $\int_{\Sigma}|A|^2e^{-f}<\infty$. Then
\begin{align}
&\int_{\Sigma}|A|^2\alpha e^{-f}d\sigma=0, \label{alpha-eq-2}\\
& \int_{\Sigma}|\nabla\alpha|^2e^{-f}d\sigma-\int_{\Sigma}\alpha^2|A|^2e^{-f}d\sigma=0.\label{alpha-eq}
\end{align}
\end{prop}

\begin{proof} Note
$L_f=\Delta_f+|A|^2+\frac12$. \eqref{delta-0} implies that
\begin{equation}\label{eq-alpha-2-1}
\Delta_f\alpha+|A|^2\alpha=0.
\end{equation}
If $\Sigma$ is closed,  we get \eqref{alpha-eq-2} and  \eqref{alpha-eq} by integrating \eqref{eq-alpha-2-1} and \eqref{delta-1},  and then using the Stokes'  formula respectively. Consider the case of non-compact  $\Sigma $. Since $|\alpha|\leq 1$ and $|\nabla\alpha|^2\leq |A|^2$,  we have 
$$ \int_{\Sigma}|\nabla\alpha|^2e^{-f}< \infty, \quad \int_{\Sigma}\alpha^2|A|^2e^{-f}<\infty.$$  Here and thereafter, for simplicity of notations, we omit $d\sigma$ in the computation of integrals.  For any positive integer $k$,
\begin{equation}\label{ine-alpha-2-1}\begin{split}
\left|\int_{\Sigma}\varphi_j(\Delta_f\alpha^k) e^{-f}\right|
&=\left|k\int_{\Sigma}\alpha^{k-1}\langle\nabla\varphi_j,\nabla\alpha\rangle e^{-f}\right|\\
&\leq k\left(\int_{\Sigma}|\nabla\varphi_j|^2e^{-f}\right)^{\frac12}\left(\int_{\Sigma}|\nabla\alpha|^2e^{-f}\right)^{\frac12}\\
&\leq2k\left(\int_{B_{j+1}\setminus B_j}e^{-f}\right)^{\frac12}\left(\int_{\Sigma}|\nabla\alpha|^2e^{-f}\right)^{\frac12}
\end{split}
\end{equation}
Since the right-hand side of \eqref{ine-alpha-2-1} tends to zero as $j\to \infty$,
\begin{equation}\label{eq-alpha-2-5}\lim_{j\rightarrow\infty}\int_{\Sigma}\varphi_j(\Delta_f\alpha^k) e^{-f}=0.
\end{equation}
Multiplying \eqref{eq-alpha-2-1} by  the  cut-off functions $\varphi_j$, we have 
\begin{equation}\label{eq-alpha-2-2}\int_{\Sigma}\varphi_j(\Delta_f\alpha) e^{-f}+\int_{\Sigma}\varphi_j|A|^2\alpha e^{-f}=0.
\end{equation}
Letting $j\to \infty$ in \eqref{eq-alpha-2-2}, by \eqref{eq-alpha-2-5} for $k=1$ and the monotone convergence theorem, we get the identity $\int_{\Sigma}|A|^2\alpha e^{-f}=0,$ that is \eqref{alpha-eq-2}.

Now we  prove \eqref{alpha-eq}. Multiplying \eqref{delta-1} by  $\varphi_j$ and integrating, we have 
\begin{align}\label{alpha-int}
\frac{1}{2}\int_{\Sigma}\varphi_j(\Delta_f\alpha^{2})e^{-f}
&=\int_{\Sigma}\varphi_j|\nabla\alpha|^{2}-\int_{\Sigma}\varphi_j|A|^{2}\alpha^{2}e^{-f}.
\end{align}
Letting $j\to\infty$ on both sides of  \eqref{alpha-int}, by \eqref{eq-alpha-2-5} for $k=2$  and  the monotone convergence theorem, we get \eqref{alpha-eq}.
 
 \end{proof}
 
 \begin{prop}\label{lem-A}
Let $\Sigma$ be a complete oriented $f$-minimal hypersurface properly  immersed in $(\mathbb{S}^{n}(\sqrt{2(n-1)})\times\mathbb{R},\bar{g}, f)$.  Assume that  $\int_{\Sigma}|A|^4e^{-f}<\infty$. Then
\begin{equation}
-\int_{\Sigma}|\nabla H|^{2}e^{-f}+\int_{\Sigma}H^{2}|A|^{2}e^{-f}+\frac14\int_{\Sigma}\alpha^{2}(1-\alpha^{2})e^{-f}=0,\label{lem-H-eq}
\end{equation}
\begin{equation}
\int_{\Sigma}|\nabla A|^{2}e^{-f}+\int_{\Sigma}|A|^{2}(\frac12-|A|^{2})e^{-f}-\frac{1}{2(n-1)}\int_{\Sigma}\alpha^{2}(1-\alpha^{2})e^{-f}=0. \label{lem-A-eq}
\end{equation}
\end{prop}
 \begin{proof} Suppose that $\Sigma$ is complete noncompact  (the compact case was considered in \cite{CMZ3}). First we prove \eqref{lem-H-eq}. By $|\alpha|\leq1$ and finite weighted volume of $\Sigma$, $$\int_{\Sigma}\alpha^{2}(1-\alpha^{2})e^{-f}<\infty.$$ 
 Note that $$\int_{\Sigma}|A|^2e^{-f}\leq \left(\int_{\Sigma}|A|^4e^{-f}\right)^{\frac12}\left(\int_{\Sigma}e^{-f}\right)^{\frac12}<\infty.$$ By $H^2\leq n|A|^2$ and  the assumptions of the proposition, $$\int_{\Sigma}H^2e^{-f}<\infty, \quad \int_{\Sigma}|A|^2H^2e^{-f}<\infty.$$
  Under  a local orthonormal frame $\{e_i\}_{i=1}^n$ on $\Sigma$, 
\begin{align*}
\nabla_{e_i} H&=
\langle\overline{\nabla}_{e_i}\overline{\nabla}f,\nu\rangle+\langle\overline{\nabla} f,\overline{\nabla}_{e_i}\nu\rangle\\
&=\overline{\nabla}^2f(e_i,\nu)+\langle\overline{\nabla} f, a_{ij}e_j\rangle\\
&=\frac12\langle e_i,\frac{\partial}{\partial h}\rangle\langle \nu,\frac{\partial}{\partial h}\rangle+\frac{a_{ij}h}{2}\langle\frac{\partial}{\partial h},e_j\rangle.
\end{align*}
Then
\begin{equation}\label{ine-grad-H}\begin{split}
|\nabla H|^2&=\displaystyle\sum_{i=1}^n |\nabla_{e_i} H|^2\\
&\leq \frac12\sum_{i=1}^n\langle e_i,\frac{\partial}{\partial h}\rangle^2\langle \nu,\frac{\partial}{\partial h}\rangle^2+\displaystyle\sum_{i,j=1}^n\frac{a_{ij}^2h^2}{2}\langle\frac{\partial}{\partial h},e_j\rangle^2\\
&\leq  \frac12\alpha^2(1-\alpha^2)+\frac{h^2}{2}|A|^2\\
&\leq \frac12\alpha^2(1-\alpha^2)+\frac14(h^4+|A|^4).
\end{split}
\end{equation}
Since $\Sigma$ has polynomial volume growth, for any positive integer $k$,
\begin{equation}\label{ine-tk}
\begin{split}
\int_{\Sigma}h^ke^{-f}&=\int_{\Sigma\cap B^{M}_{r_{0}}(p)}h^ke^{-\frac{h^2}{4}}+\sum_{i=0}^{\infty}\int_{\Sigma\cap(B^{M}_{r_{0}+i+1}(p)\backslash B^{M}_{r_{0}+i}(p))}h^ke^{-\frac{h^2}{4}}\\
&\leq C_{1}\text{Vol}(\Sigma\cap B^{M}_{r_{0}}(p))\\
&\quad+C\sum_{i=0}^{\infty}(r_0+i+1)^ke^{-\frac{1}{4}(r_{0}+i)^{2}}\text{Vol}(\Sigma\cap B_{r_{0}+i+1}^{M}(p))\\
&\leq C\bigr[r_{0}^{d}+\sum_{i=0}^{\infty}e^{-\frac{1}{4}(r_{0}+i-c)^{2}}(r_{0}+i+1)^{d+k}\bigr]\\
&< \infty.
\end{split}
\end{equation}
 \eqref{ine-grad-H}, \eqref{ine-tk} and  the assumption of the proposition imply that $$\int_{\Sigma}|\nabla H|^2e^{-f}<\infty.$$
Multiplying \eqref{delta-2} by  $\varphi_j^2$, we have 
\begin{align}
\frac12\int_{\Sigma}\varphi_j^2(\Delta_fH^{2})e^{-f}&=\int_{\Sigma}\varphi_j^2|\nabla H|^{2}e^{-f}-\int_{\Sigma}\varphi_j^2|A|^{2}H^2e^{-f}\nonumber\\
&\quad-\frac12\int_{\Sigma}\varphi_j^2H^{2}e^{-f}+\frac12\int_{\Sigma}\varphi_j^2\langle\nabla\alpha^{2},\nabla f\rangle e^{-f}.\label{eq-H-11}
\end{align}
Note
\begin{equation}\label{ine-Hh-01}\begin{split}
\left|\int_{\Sigma}\varphi_j^2(\Delta_fH^{2})e^{-f}\right|
&=4\left|\int_{\Sigma}H\varphi_j\langle\nabla\varphi_j,\nabla H\rangle e^{-f}\right|\\
&\leq 4\left(\int_{\Sigma}|\nabla\varphi_j|^2H^2e^{-f}\right)^{\frac12}\left(\int_{\Sigma}\varphi_j^2|\nabla H|^2e^{-f}\right)^{\frac12}\\
&\leq 4\left(\int_{B_{j+1}\setminus B_j}H^2e^{-f}\right)^{\frac12}\left(\int_{\Sigma}|\nabla H|^2e^{-f}\right)^{\frac12}.\end{split}
\end{equation}
 Since the right-hand side of \eqref{ine-Hh-01} tends to zero as $j\to\infty$,  
 \begin{equation}\label{eq-H-2-8}
 \displaystyle\lim_{j\to\infty}\int_{\Sigma}\varphi_j^2(\Delta_fH^{2})e^{-f}=0.
 \end{equation}
Observe that
\begin{equation}\label{eqs9}\begin{split}
\Delta_ff&=\sum_{i=1}^n[(\overline{\nabla}^2f)_{ii}-a_{ii}f_{\nu}]-\langle\nabla f,\nabla f\rangle\\
&=\frac12\sum_{i=1}^{n}\langle e_{i},\frac{\partial}{\partial h}\rangle^{2}-H{f}_\nu-|\nabla f|^2\\
&=\frac12(1-\alpha^{2})-|\overline{\nabla}f|^{2}\\
&=\frac12(1-\alpha^{2})-f.
\end{split}
\end{equation}
By Stokes' formula,  \eqref{eqs9} and $H^2=\left(\frac{h\alpha}{2}\right)^2=\alpha^2f$, we have
\begin{equation}\label{grad-f-alpha}\begin{split}
&\quad\frac12\int_{\Sigma}\varphi_j^2\langle\nabla\alpha^{2},\nabla f\rangle e^{-f}\\
&=-\frac12\int_{\Sigma}\varphi_j^2\alpha^{2}(\Delta_{f}f)e^{-f}-\frac12\int_{\Sigma}\alpha^{2}\langle\nabla\varphi_j^2,\nabla f\rangle e^{-f}\\
&=-\frac14\int_{\Sigma}\varphi_j^2\alpha^{2}(1-\alpha^{2})e^{-f}+\frac12\int_{\Sigma}\varphi_j^2 H^{2}e^{-f}\\
&\quad -\int_{\Sigma}\varphi_j\alpha^{2}\langle\nabla\varphi_j,\nabla f\rangle e^{-f}.
\end{split}
\end{equation}
 By $|\nabla f|\leq |\overline{\nabla}f|\leq \frac{h}{2}$,
 \begin{align*}
\left(\int_{\Sigma}\varphi_j\alpha^2\langle\nabla\varphi_j,\nabla f\rangle e^{-f}\right)^2&\leq\left(\int_{\Sigma}|\nabla\varphi_j|^2e^{-f}\right)\left(\int_{\Sigma}\varphi_j^2|\nabla f|^2e^{-f}\right)\nonumber\\
  &\leq\left(\int_{B_{j+1}\setminus B_j}e^{-f}\right)\left(\int_{\Sigma}h^2e^{-f}\right).
   \end{align*}
This implies that 
\begin{equation}\label{ine-grad-var}
\displaystyle\lim_{j\to\infty}\int_{\Sigma}\varphi_j\alpha^2\langle\nabla\varphi_j,\nabla f\rangle e^{-f}=0. 
\end{equation}
Letting $j\to\infty$ on the both side of \eqref{grad-f-alpha},  by \eqref{ine-grad-var} and
 the monotone convergence theorem, we have
\begin{equation}\label{eq-varphi-alpha-1}
\lim_{j\rightarrow \infty}\frac12\int_{\Sigma}\varphi_j^2\langle\nabla\alpha^{2},\nabla f\rangle e^{-f}=-\frac14\int_{\Sigma}\alpha^{2}(1-\alpha^{2})e^{-f}+\frac12\int_{\Sigma} H^{2}e^{-f}.
\end{equation}
Letting $j\to\infty$ on the both side of \eqref{eq-H-11} and using \eqref{eq-H-2-8}, \eqref{eq-varphi-alpha-1} and  the monotone convergence theorem, 
 we get \eqref{lem-H-eq}, that is
\begin{equation*}
-\int_{\Sigma}|\nabla H|^{2}e^{-f}+\int_{\Sigma}H^{2}|A|^{2}e^{-f}+\frac14\int_{\Sigma}\alpha^{2}(1-\alpha^{2})e^{-f}=0.
\end{equation*}

Now we prove \eqref{lem-A-eq}. Multiplying \eqref{delta-3} by  $\varphi_j^2$, we have
\begin{equation}\label{eq-Aa-1}
\begin{split}
&\quad\frac{1}{2}\int_{\Sigma}\varphi_j^2\Delta_f|A|^{2}e^{-f}\\
&=\int_{\Sigma}\varphi_j^2|\nabla A|^{2}e^{-f}+\frac12\int_{\Sigma}\varphi_j^2|A|^{2}e^{-f}-\int_{\Sigma}\varphi_j^2|A|^{4}e^{-f}\\
& \quad -\frac{1}{n-1}\int_{\Sigma}\varphi_j^2|\nabla\alpha|^{2}e^{-f}+\frac{1}{n-1}\int_{\Sigma}\varphi_j^2\alpha^{2}|A|^{2}e^{-f}\\
&\quad  -\frac{1}{n-1}\int_{\Sigma}\varphi_j^2\alpha^{2}fe^{-f}+\frac{1}{n-1}\int_{\Sigma}\varphi_j^2\langle\nabla\alpha^{2},\nabla f\rangle e^{-f}.
\end{split}
\end{equation}
Using $|\nabla|A||\leq |\nabla A|$, 
\begin{align}\label{ine-Aa-2}
\frac{1}{2}\int_{\Sigma}\varphi_j^2\Delta_f|A|^{2}e^{-f}&=-2\int_{\Sigma}\langle\varphi_j\nabla\varphi_j, |A|\nabla|A|\rangle e^{-f}\nonumber\\
&\leq  \epsilon\int_{\Sigma}\varphi_j^2|\nabla|A||^2e^{-f}+\frac{1}{\epsilon}\int_{\Sigma}|A|^2|\nabla\varphi_j|^2e^{-f}\\
&\leq  \epsilon\int_{\Sigma}\varphi_j^2|\nabla A|^2e^{-f}+\frac{1}{\epsilon}\int_{\Sigma}|A|^2|\nabla\varphi_j|^2e^{-f}.\nonumber
\end{align}
Substitute \eqref{ine-Aa-2} into \eqref{eq-Aa-1}. We have
\begin{equation}\label{ine-Aa-3}
\begin{split}
&\qquad(1-\epsilon)\int_{\Sigma}\varphi_j^2|\nabla A|^{2}e^{-f}\\
&\quad\leq \int_{\Sigma}\varphi_j^2|A|^{4}e^{-f}
+\frac{1}{\epsilon}\int_{\Sigma}|A|^2|\nabla\varphi_j|^2e^{-f}\\
&\qquad +\frac{1}{n-1}\int_{\Sigma}\varphi_j^2|\nabla\alpha|^{2}e^{-f}-\frac{1}{n-1}\int_{\Sigma}\varphi_j^2\alpha^{2}|A|^{2}e^{-f}\\
&\qquad +\frac{1}{n-1}\int_{\Sigma}\varphi_j^2\alpha^{2}fe^{-f}-\frac{1}{n-1}\int_{\Sigma}\varphi_j^2\langle\nabla\alpha^{2},\nabla f\rangle e^{-f}.
\end{split}
\end{equation}
 Observe that  $\int_{\Sigma}\alpha^2fe^{-f}=\int_{\Sigma}H^2e^{-f}<\infty$ and all terms on   the right-hand side of \eqref{ine-Aa-3} converge as $j\to\infty$.  By the monotone convergence theorem, 
$$\int_{\Sigma}|\nabla A|^{2}=\displaystyle\lim_{j\to\infty}\int_{\Sigma}\varphi_j^2|\nabla A|^{2}<\infty.$$
 Furthermore,
 \begin{align}\label{ine-Aa-01}
\left|\frac12\int_{\Sigma}\varphi_j^2(\Delta_f|A|^{2})e^{-f}\right|
&=2\left|\int_{\Sigma}|A|\varphi_j\langle\nabla\varphi_j,\nabla |A|\rangle e^{-f}\right|\nonumber\\
&\leq 2\left(\int_{\Sigma}|\nabla\varphi_j|^2|A|^2e^{-f}\right)^{\frac12}\left(\int_{\Sigma}\varphi_j^2|\nabla |A||^2e^{-f}\right)^{\frac12}\nonumber\\
&\leq 4\left(\int_{B_{j+1}\setminus B_j}|A|^2e^{-f}\right)^{\frac12}\left(\int_{\Sigma}|\nabla A|^2e^{-f}\right)^{\frac12}.
\end{align}
 Since the right-hand side of \eqref{ine-Aa-01} tends to zero as $j\to\infty$,  
\begin{equation} \label{limit-1}
\displaystyle\lim_{j\to\infty}\int_{\Sigma}\varphi_j^2(\Delta_f|A|^{2})e^{-f}=0.
\end{equation}
 Letting $j\to\infty$ on both sides of  \eqref{eq-Aa-1} and using  the monotone convergence theorem, \eqref{limit-1},  \eqref{alpha-eq},  \eqref{eq-varphi-alpha-1} and $H^2=\alpha^2f$, we get \eqref{lem-A-eq}.

\end{proof}

The integral identities in Proposition \ref{lem-A} can be used to classify complete totally geodesic $f$-minimal hypersurfaces.

\medskip

\noindent {\it Proof of Theorem \ref{lem-geodesic}}. Examples \ref{ex-fmin-1} and \ref{ex-fmin-2} say that $f$-minimal hypersurfaces $\mathbb{S}^n(\sqrt{2(n-1)})\times\{0\}$ and $\mathbb{S}^{n-1}(\sqrt{2(n-1)})\times\mathbb{R}$ are totally geodesic. Now we prove the inverse. Since $A\equiv 0$ on $\Sigma$,
by   \eqref{lem-H-eq}, we have $$\int_{\Sigma}\alpha^2(1-\alpha^2)e^{-f}=0.$$ 
Note $|\alpha|\leq 1$. $\alpha^2(1-\alpha^2)\equiv 0$. Hence
  either $\alpha\equiv 0$ or $\alpha^2\equiv 1$ on $\Sigma$ . 
  
 (i) the case of $\alpha^2\equiv 1$.  Without lost of generality,  suppose $\alpha\equiv 1$. This means that  $\nu=\frac{\partial}{\partial h}$ on $\Sigma$. Hence $\Sigma$ must be in a horizontal slice $\mathbb{S}^{n}(\sqrt{2(n-1)})\times\{h\}$. By Example 1,
$\Sigma$ must be  $\mathbb{S}^{n}(\sqrt{2(n-1)})\times\{0\}$. 

(ii) the case of  $\alpha\equiv 0$. In this case, $\frac{\partial}{\partial h}\in T_p\Sigma$ for any $p\in \Sigma$. This implies that any vertical line $\{x\}\times \mathbb{R}$ passing through $\Sigma$ must be a curve in $\Sigma$. Thus $\Sigma=\Sigma_1\times\mathbb{R}$, where $\Sigma_1\subset\mathbb{S}^{n}(\sqrt{2(n-1)})$. Let $\pi:\mathbb{S}^{n}(\sqrt{2(n-1)})\times \mathbb{R}\to \mathbb{S}^{n}(\sqrt{2(n-1)})$ denote the projection onto the first factor $\mathbb{S}^{n}(\sqrt{2(n-1)})$. Since the rank of the differential  $d\pi|_{\Sigma}$ is $n-1$,   $\Sigma_1$ is an $(n-1)$-dimensional hypersurface in  $\mathbb{S}^{n}(\sqrt{2(n-1)})$. Like Example \ref{ex-fmin-2},  the unit normal $\nu$ of $\Sigma$  is the unit normal of $\Sigma_1$. $\Sigma$ is $f$-minimal in $\mathbb{S}^{n}(\sqrt{2(n-1)})\times \mathbb{R}$ if and only if $\Sigma_1$ is minimal in $\mathbb{S}^{n}(\sqrt{2(n-1)})$.  In addition, $\Sigma$ is properly immersed if and only if $\Sigma_1$ is closed.  Furthermore, since $\overline{\nabla}_{\frac{\partial}{\partial h}}\nu=0$,  $A(x,h)=A_{\Sigma_1}(x), (x,h)\in \Sigma_1\times \mathbb{R}$. Hence $\Sigma_1$ is totally geodesic in $\mathbb{S}^{n}(\sqrt{2(n-1)})$. Therefore $\Sigma_1$ is  $\mathbb{S}^{n-1}(\sqrt{2(n-1)})$ and thus $\Sigma$ is $ \mathbb{S}^{n-1}(\sqrt{2(n-1)})\times\mathbb{R}$.

\qed

\section{$L_f$-index of  complete $f$-minimal hypersurfaces}\label{sec-index}

Our purpose of  this section is to prove Theorem \ref{th-index}. First we prove that  $\mathbb{S}^{n-1}(\sqrt{2(n-1)})\times\mathbb{R}\subset (\mathbb{S}^{n}(\sqrt{2(n-1)})\times\mathbb{R},\bar{g},f)$ has $L_f$-index one.

\begin{prop}\label{index-S-(n-1)}
The $f$-minimal hypersurface  $\mathbb{S}^{n-1}(\sqrt{2(n-1)})\times\mathbb{R}$ in the  cylinder shrinking soliton $(\mathbb{S}^{n}(\sqrt{2(n-1)})\times\mathbb{R},\bar{g},f)$ satisfies that $L_f$-ind $(\mathbb{S}^{n-1}(\sqrt{2(n-1)})\times\mathbb{R})=1$.
\end{prop}
\begin{proof} For $\Sigma=\mathbb{S}^{n-1}(\sqrt{2(n-1)})\times\mathbb{R}$, the normal $\nu\in T\mathbb{S}^{n}(\sqrt{2(n-1)})$,
$\nabla f=(\overline{\nabla}f)^{\top}=\frac{h}{2} \frac{\partial}{\partial h}$, and $|A|^{2}=0$.
Hence, 
\begin{equation}
L_{f}=\Delta-\frac h2 \langle  \frac{\partial}{\partial h}, \nabla\cdot\rangle +\frac12.
\end{equation}
Let $\psi(x), x\in \mathbb{S}^{n-1}(\sqrt{2(n-1)})$ and $\rho(h), h\in \mathbb{R}$ satisfy the following eigenvalue problems respectively:
\begin{align}
\Delta_{\mathbb{S}^{n-1}(\sqrt{2(n-1)})}\psi(x)&=-\lambda\psi(x),\\
 \frac{d^2\rho}{dh^2}(h)-\frac h2 \frac{d\rho}{dh}(h)&=-\eta\rho(h).\label{her}
 \end{align}
  Let $\{e_i\}, i=1,\cdots, n-1,$ denotes the orthonormal frame on $\mathbb{S}^{n-1}(\sqrt{2(n-1)})$. Then 
\begin{equation}\label{eq-eigen-1}\begin{split}
L_{f}\psi(x)\rho(h)&=\sum_{i=1}^{n-1}\nabla_i\nabla_i(\psi(x)\rho(h))+\nabla_{\frac{\partial}{\partial h}}\nabla_{\frac{\partial}{\partial h}}(\psi(x)\rho(h))\\
&\quad -\frac h2 \langle  \frac{\partial}{\partial h}, \nabla(\psi(x)\rho(h))\rangle +\frac12\psi(x)\rho(h)\\
&=\left(\Delta_{\mathbb{S}^{n-1}(\sqrt{2(n-1)})}\psi(x)\right)\rho(h)+\frac12\psi(x)\rho(h)\\
&\quad+\psi(x)\left( \frac{d^2\rho}{dh^2}(h)-\frac h2 \frac{d\rho}{dh}(h)\right)\\
&=(-\lambda-\eta+\frac12)\psi(x)\rho(h).
\end{split}
\end{equation}

\noindent  It is known that the eigenvalues  of the Laplacian $\Delta_{\mathbb{S}^{n-1}(\sqrt{2(n-1)})}$, counted with multiplicity,  are 
\begin{equation*}
\{\lambda_k: k=0,1,2, \ldots,\}=\{0,  \underbrace{\hbox{$\frac{1}{2},\ldots, \frac12$}}_{\hbox{n}}, \frac{n}{n-1}, \ldots\}
\end{equation*}
and there exists a complete orthonormal system for space $L^2( d\sigma_{\mathbb{S}^{n-1}(\sqrt{2(n-1)})})$ consisting of eigenfunctions $\psi_k$ of $\Delta_{\mathbb{S}^{n-1}(\sqrt{2(n-1)})}$ associated to $\lambda_k$.  On the other hand, for the operator $\frac{d^2}{dh^2}-\frac{h}{2}\frac{d}{dh}, h\in \rr$, it is known that its spectrum  on $L^2(\rr,e^{-\frac{h^2}{4}}dh)$ is discrete, the eigenvalues of the operator $\frac{d^2}{dh^2}-\frac{h}{2}\frac{d}{dh}, h\in \rr$, counted by multiplicity, are 
\begin{equation*}
\{\eta_l: l=0,1,2, \ldots,\}=\{0,  \frac12, 1, \ldots\},
\end{equation*}
 and the so-called Hermite polynomials $\rho_l(h)$ are  orthonormal eigenfunctions associated to $\eta_l$,  which form a complete orthonormal system for the weighted space $L^2(\rr, e^{-\frac{h^2}{4}}dh)$.  Through a standard argument, one can verify  that $\{\psi_k(x)\rho_l(h)\}$ is a complete orthonormal system for  space $L^2(\Sigma, e^{-f}d\sigma)$,  where $d\sigma=d\sigma_{\mathbb{S}^{n-1}(\sqrt{2(n-1)})}dh$, and  $\psi_k(x)\rho_l(h)$ are the eigenfunctions associated to  the eigenvalues counted with multiplicity:
\begin{equation*}
\mu_{k,l}=\lambda_{k}+\eta_l-\frac12.
\end{equation*}
Note that
\begin{align*}
\mu_{0,0}&=-\frac12,\\
\mu_{k,l}&\geq 0\qquad \textrm{for all other} 
\quad k, l,\\
 \mu_{k,l} &\rightarrow +\infty.
\end{align*}
Hence the spectrum of the operator $L_f$ is discrete. $L_{f}$ has only a  negative eigenvalue $\mu_{0,0}$ with multiplicity one and the associated eigenfunction $\psi_0\rho_0\equiv 2\sqrt{\pi}$.  By the variational characterization of eigenvalues, we have that for any $\varphi\in \{\psi_0\rho_0\}^{\perp}\cap C_0^{\infty}(\Sigma)$,  $Q(\varphi,\varphi)\geq 0$. Hence by Proposition \ref{mu1},  the $L_f$-index of  $\mathbb{S}^{n-1}(\sqrt{2(n-1)})\times\mathbb{R}$ is $1$.

\end{proof}

\begin{remark} In the proof of Proposition \ref{index-S-(n-1)}, the discreteness of $L_f$ can also be obtained by Corollary 3 in \cite{CZ2}, since  $\mathbb{S}^{n-1}(\sqrt{2(n-1)})\times\mathbb{R}$ is totally geodesic and $L_f=\Delta_f+\frac12$. The discreteness of the spectrum of $\Delta_f$  implies that the spectrum of $\Delta_f+\frac12$ is also discrete.
\end{remark}

\begin{lemma}\label{A-finite} Let $\Sigma$ be a complete oriented properly immersed $f$-minimal hypersurface   in $(\mathbb{S}^{n}(\sqrt{2(n-1)})\times\mathbb{R},\bar{g},f)$.  If $\mu_1(L_f)\neq -\infty$, then $$\int_{\Sigma}|A|^2e^{-f}<\infty.$$
\end{lemma}
\begin{proof} Let  $\varphi_j$  be the cut-off functions as before. So $\varphi_j\in W^{1,2}(e^{-f}d\sigma)$. Since $\mu_1(L_f)\neq -\infty$, by Proposition \ref{eigenfunction1}, there is a $C^2$ positive function $h$ on $\Sigma$ satisfying $L_fh=-\mu_1(L_f)h$. By Proposition \ref{logf}, we have 
\begin{equation}
\int_{\Sigma}\varphi_j^2|A|^2e^{-f}\leq \int_{\Sigma}(4|\nabla\varphi_j|^2-2\mu_1(L_f)\varphi_j^2)e^{-f}.
\end{equation}
So 
\begin{equation}
\int_{B_j}|A|^2e^{-f}\leq \int_{\Sigma}(4+2|\mu_1(L_f)|)e^{-f}<\infty.
\end{equation}
Letting $j\rightarrow\infty$, we obtain the conclusion.
\end{proof}

Now we are ready to prove Theorem \ref{th-index}, which classifies the  complete noncompact $f$-minimal hypersurfaces in $(\mathbb{S}^{n}(\sqrt{2(n-1)})\times\mathbb{R},\bar{g},f)$ of $L_f$-index one.
\medskip

\noindent {\it Proof of Theorem \ref{th-index}.}  It was   proved  in \cite{CMZ2}  that there is no complete stable $f$-minimal hypersurface  with finite weighted volume, immersed in $M^{n+1}$ with $\overline{\Ric}_f\geq \frac 12$. Hence,  Proposition \ref{prop-1-cor} implies that $L_f$-$\textrm{ind}(\Sigma)\geq 1$.
By Lemma 2 in \cite{CMZ3} and Proposition \ref{index-S-(n-1)}, $\mathbb{S}^{n}(\sqrt{2(n-1)})\times\{0\}$ and $\mathbb{S}^{n-1}(\sqrt{2(n-1)})\times\mathbb{R}$ have $L_f$-index one. Now we prove the converse.   Suppose that $L_f$-$\textrm{ind}(\Sigma)= 1$.  

In the case of   $\alpha\not\equiv 0$ on $\Sigma$.  By \eqref{delta-0},  $$L_f\alpha=\frac12\alpha.$$ 
 Note $\int_{\Sigma}\alpha^2e^{-f}\leq \int_{\Sigma} e^{-f}<\infty$.
Hence $\alpha$ is an $ L^{2}(e^{-f}d\sigma)$ eigenfunction  with  negative eigenvalue $-\frac12$.  Since $L_f$-index of $\Sigma$ is one, by Proposition \ref{mu1},  the bottom  $\mu_1(L_f)$  of the spectrum of $L_f$ satisfies that   $\mu_1(L_f)= -\frac12$.   By $|\nabla\alpha|^2\leq |A|^2$ and Lemma \ref{A-finite},  $$\int_{\Sigma}|\nabla\alpha|^2e^{-f}\leq \int_{\Sigma}|A|^2e^{-f}<\infty.$$ 
Hence $\alpha$ is in $W^{1,2}(e^{-f}d\sigma)$.
By Proposition \ref{eigenfunction1}, $\alpha>0$ on $\Sigma$ without lost of  generality. 
Note  the integrability of $|A|^2$ implies that  \eqref{alpha-eq-2} holds on $\Sigma$, that is, $\int_{\Sigma}|A|^2\alpha e^{-f}=0.$
Thus $|A|\equiv 0$ on $\Sigma$.  Note $\mathbb{S}^{n-1}(\sqrt{2(n-1)})\times\mathbb{R}$ has $\alpha\equiv 0$. By Theorem \ref{lem-geodesic}, $\Sigma$ must be $\mathbb{S}^n(\sqrt{2(n-1)})\times\{0\}$. 

In the case of $\alpha\equiv 0$ on $\Sigma$. By the proof of Theorem \ref{lem-geodesic} and Example \ref{ex-fmin-2},    $\Sigma$ is $\Sigma_1\times\mathbb{R}$,  where $\Sigma_1$ is a closed minimal hypersurface in  $\mathbb{S}^n(\sqrt{2(n-1)})$. Note 
\begin{equation}\label{Lf-product}\begin{split}
L_f&=\Delta_f+|A|^2+\frac{1}{2}\\
&=\Delta+\langle\nabla f,\nabla\cdot\rangle+|A_{\Sigma_1}|^2+\frac12\\
&= \left(\Delta_{\Sigma_1}+|A_{\Sigma_1}|^2+\overline{\text{Ric}}_{\mathbb{S}^{n}(\sqrt{2(n-1)})}(\nu,\nu)\right)+\left(\frac{\partial^2}{\partial h^2}+\frac12\frac{\partial}{\partial h}\right). 
\end{split}
\end{equation}
Observe that the first part of the right-hand side of \eqref{Lf-product} is just the Jacobi operator $J_{\Sigma_1}$ for minimal hypersurface $\Sigma_1$. If the index of $\Sigma_1$  is more than $1$, then  there exist at least two linearly independent eigenfunctions $\psi_i(x), i=1,2$ of $J_{\Sigma_1}$ corresponding to negative eigenvalues. Obviously $\int_{\Sigma}\psi_i(x)e^{-\frac{h^2}{4}}d\sigma<\infty$.  Thus $\psi_1(x)$ and $\psi_2(x)$ are the two linearly independent $L^2(e^{-f}d\sigma)$ eigenfunctions of $L_f$ associated  to  negative eigenvalues. By Proposition \ref{mu1},   $L_f$-ind($\Sigma$)$\geq 2$, which contradicts the assumption. Hence the index of $\Sigma_1$ must less than $2$.  It is known that in round sphere $\mathbb{S}^{n}(\sqrt{2(n-1)})$, there is no closed stable minimal hypersurface  and a  closed index one minimal hypersurface must be the totally geodesic spheres $\mathbb{S}^{n-1}(\sqrt{2(n-1)})$ (\cite{S}).  Therefore $\Sigma_1=\mathbb{S}^{n-1}(\sqrt{2(n-1)})$ and so $\Sigma=\mathbb{S}^{n-1}(\sqrt{2(n-1)})\times \mathbb{R}$.

 \qed

\begin{remark} There are weaker conditions on the non-existence of $L_f$-stable $f$-minimal hypersurfaces in a manifold with $\overline{Ric}_f\geq \frac12$. In fact, when $f$-minimal hypersurface $\Sigma^{n}$ has the weighted volume growth satisfying   $\displaystyle\lim_{r\rightarrow\infty}\frac{\log V_f^{\Sigma}(r)}{r}=\beta, \beta<\sqrt{2}$, using  an estimate of the bottom $\mu_1(\Delta_f)$ of the spectrum   of $\Delta_f$ by \cite{B}, \cite{MW},  one can get that $L_f$-$\textrm{ind}(\Sigma)\geq 1$. This result was proved by Impera and Rimoldi in \cite{IR} (Theorem 3.9 \cite{IR}). Since we only studied properly immersed $f$-minimal hypersurfaces, we prefer to let  the first part of Theorem \ref{th-index} as stated.
\end{remark}

%\section{$f$-minimal hypersurfaces in $\mathbb{S}^n\times\mathbb{R}$}

%%%%%%%%%%%%%%%%%SECTION%%%%%%%%%%%%%%%%%%%%%%%%%%

\section{Pinching  theorem }\label{sec-pinch}

Recall  the gap phenomenon for the compact minimal hypersurfaces in sphere was first studied in \cite{S} 
by using the integral identity derived from Simons' equation.
In this section, we will prove a pinching theorem and subsequently obtain a gap phenomenon for the norm of the second fundamental form of complete $f$-minimal hypersurfaces in $(\mathbb{S}^{n}(\sqrt{2(n-1)})\times \mathbb{R}, \overline{g},f)$.

\medskip
\noindent {\it Proof of Theorem \ref{th-pinch}.  We only need to consider the case when $\Sigma $ is  non-compact. \eqref{pinch-1} says that $|A|$ is bounded. So $\int_{\Sigma}|A|^2e^{-f}<\infty$ and $ \int_{\Sigma}|A|^4e^{-f}<\infty$. 
Observe that the assumption 
\begin{eqnarray*}
 \left| |A|^2-\frac14\right|\leq \frac14\left(\sqrt{1-\frac{8}{n-1}\alpha^2(1-\alpha^2)}\right)
\end{eqnarray*}
is equivalent to
\begin{equation}\label{equiv-1}\begin{split}
&\quad |A|^{2}(\frac{1}{2}-|A|^{2})-\frac{1}{2(n-1)}\alpha^{2}(1-\alpha^{2})\\
&=-(|A|^2-\frac14)^2+(\frac{1}{4})^2[1-\frac{8}{n-1}\alpha^2(1-\alpha^2)]\geq 0.
\end{split}
\end{equation}
\eqref{equiv-1} and  \eqref{lem-A-eq} imply that  for $n\geq3$, on $\Sigma$
\begin{equation*}
\nabla A\equiv0,
\end{equation*}
and 
\begin{equation}\label{rigi-A}
|A|^{2}(\frac{1}{2}-|A|^{2})-\frac{1}{2(n-1)}\alpha^{2}(1-\alpha^{2})=0.
\end{equation}
Hence, $|A|^{2}$ and $H$ are constants. By  \eqref{lem-H-eq}, we have
\begin{equation*}
\int_{\Sigma}|A|^{2}H^{2}e^{-f}+\frac{1}{4}\int_{\Sigma}\alpha^{2}(1-\alpha^{2})e^{-f}=0.
\end{equation*}
This implies that
\begin{equation*}
\alpha^{2}(1-\alpha^{2})\equiv 0\quad \text{and} \quad H\equiv 0.
\end{equation*}
Then  on $\Sigma$,
\begin{equation*}
\alpha\equiv0\quad\emph{or}\quad\alpha^2\equiv 1.
\end{equation*}

If $\alpha^2\equiv 1$, we may assume that $\nu\equiv \frac{\partial}{\partial h}$. By Example 1, $\Sigma$ must be $\mathbb{S}^{n}(\sqrt{2(n-1)})\times\{0\}$.

If $\alpha^2\equiv 0$, as we have shown in the proof of Theorem \ref{th-index}, $\Sigma$ is $\Sigma_1\times \mathbb{R}$, where $\Sigma_1$ be a closed minimal hypersurface in  $\mathbb{S}^{n}(\sqrt{2(n-1)})$.  Moreover $\nu$  is the normal of $\Sigma_1$. Since $|A|^2$ is constant on $\Sigma$, by \eqref{rigi-A}, $|A|\equiv 0$ or $|A|^2\equiv\frac12$ on $\Sigma$. If $A\equiv 0$, by Theorem \ref{lem-geodesic},  $\Sigma$ must be $\mathbb{S}^{n-1}(\sqrt{2(n-1)})\times\mathbb{R}$. If $|A|^2\equiv\frac12$,  since $\frac{\partial}{\partial h}$ is parallel,  the second fundamental form $A_{\Sigma_1}$ of $\Sigma_1$ satisfies that $|A_{\Sigma_1}|^2=\frac12$.  By the result on minimal hypersurfaces in sphere by Chern-do Carmo-Kobayashi \cite{CdCK} and Lawson \cite{L},  $\Sigma_1$ must be minimal clifford torus $T=\mathbb{S}^k(\sqrt{2k})\times\mathbb{S}^l(\sqrt{2l})$, where $k+l=n-1$. Hence $\Sigma$ is $T\times\mathbb{R}$.

\qed

Corollary \ref{cor-gap} is a straightforward consequence of Theorem \ref{th-pinch} since
$$4\alpha^2(1-\alpha^2)\leq 1.$$

\end{document}